\documentclass[11pt]{amsart}
\usepackage{mathtools} 
\usepackage{amsthm,amsfonts,amssymb,amsrefs}
\usepackage{bm} 
\usepackage{hyperref}
\usepackage{todonotes}
\usepackage{tikz}
\usepackage{graphicx}
\usepackage{hyperref}
\usepackage{comment} 
\usepackage{tabularx}

\newcommand{\apref}[3]{\hyperref[#2]{#1\ref*{#2}#3}}

\usepackage{enumerate}
\usepackage[latin1]{inputenc}

\usepackage{amssymb,amsthm,amsmath}

\input{xy}
\xyoption{all}

\usepackage{mathrsfs}

\DeclareMathOperator{\arcosh}{arcosh}

\DeclareMathOperator{\PSL}{PSL}

\DeclareMathOperator{\Ima}{Im}
\DeclareMathOperator{\Rea}{Re}

\DeclareMathOperator{\red}{Red}

\newcommand\N{\mathbb{N}}

\newcommand\R{\mathbb{R}}
\newcommand\Z{\mathbb{Z}}
\newcommand\C{\mathbb{C}}

\DeclareMathOperator{\vol}{vol}

\def\beq{\begin{eqnarray}}
\def\eeq{\end{eqnarray}}
\newcommand{\nn}{\nonumber}

\usepackage{color}

\definecolor{pink}{rgb}{1,0,1}

\newtheorem{conjecture}{Conjecture}

\usepackage{tikz}
\usetikzlibrary{angles,quotes,calc,shapes}
\newcommand{\RN}[1]{%
	\textup{\uppercase\expandafter{\romannumeral#1}}%
}

\usepackage{tikz}
\usetikzlibrary{decorations.markings}
\usetikzlibrary{angles,quotes,calc}
\usetikzlibrary{calc}
\usepackage{subfigure}

\definecolor{blue}{rgb}{0,0,1}
\definecolor{red}{rgb}{1,0,0}
\definecolor{green}{rgb}{0,.6,.2}
\definecolor{purple}{rgb}{1,0,1}
\definecolor{brown}{rgb}{.59,.29,0}

\long\def\red#1\endred{{\color{red}#1}}
\long\def\blue#1\endblue{{\color{blue}#1}}
\long\def\purple#1\endpurple{{\color{purple} #1}}
\long\def\green#1\endgreen{{\color{green}#1}}
\long\def\brown#1\endbrown{{\color{brown}#1}}

\newcommand{\HH}{\mathbb{H}}

\DeclareMathOperator{\sech}{sech}
\DeclareMathOperator{\csch}{csch}
\usepackage{enumerate}
\usepackage[latin1]{inputenc}
\usepackage{dsfont}
\usepackage{amssymb,amsthm,amsmath}

\definecolor{pink}{rgb}{1,0,1}     
\definecolor{blue}{rgb}{0,0,1}
\definecolor{red}{rgb}{1,0,0}
\definecolor{green}{rgb}{0,.6,.2}
\definecolor{purple}{rgb}{1,0,1}
\definecolor{brown}{rgb}{.59,.29,0}

\hypersetup{colorlinks=true,allcolors=black}

\numberwithin{equation}{section}

\theoremstyle{definition}

\newtheorem{thm}{Theorem}[section]
\newtheorem*{thm*}{Theorem}
\newtheorem{defn}[thm]{Definition}

\newtheorem*{prop*}{Proposition}

\newtheorem{lemma}[thm]{Lemma}

\newcommand{\pmbK}{\pmb{K}}

\newcommand{\cR}{\mathcal{R}}
\renewcommand{\vol}{\operatorname{vol}}
\newcommand{\cN}{\mathcal{N}}

\definecolor{pink}{rgb}{1,0,1}

\title[Casimir energy of singular orbifolds]{Casimir energy of hyperbolic orbifolds with conical singularities}
\author[K.\@ Fedosova]{Ksenia Fedosova}
\address{Mathematisches Institut, University of M\"unster, 
Einsteinstr. 62, 48149 M\"unster, Germany}
\email{ksenia.fedosova@uni-muenster.de}

\author[J.\@ Rowlett]{Julie Rowlett}
\address{Mathematical Sciences, Chalmers University of Technology and University
of Gothenburg, 412 96 Gothenburg, Sweden}
\email{julie.rowlett@chalmers.se}

\author[G.\@ Zhang]{Genkai Zhang}
\address{Mathematical Sciences, Chalmers University of Technology and University of Gothenburg, 412 96 Gothenburg, Sweden}
\email{genkai@chalmers.se}

\begin{document}

\maketitle
\begin{abstract} 
	In this article, we obtain the explicit expression of the Casimir energy for 2-dimensional Clifford-Klein space forms in terms of the geometrical data of the underlying spacetime with the help of zeta-regularization techniques.  The spacetime is geometrically expressed as a compact hyperbolic orbifold surface that may have finitely many   conical singularities. 
 In computing the contribution to the energy from a conical singularity, we derive an expression of an elliptic orbital integral as an infinite sum of special functions.  We prove that this sum converges exponentially fast.  
 Additionally, we  show that under a natural assumption (known to hold asymptotically) on the growth of the lengths of primitive closed geodesics of the $(2, 3, 7)$-triangle group orbifold its Casimir energy is positive (repulsive).
\end{abstract}


\section{Introduction} \label{sec:intro} 
The Casimir energy is named after the Dutch physicist, Hendrik B. G. Casimir who showed in 1948 that two uncharged parallel metal plates alter the vacuum fluctuations in such a way as to attract each other.  This is now referred to as \em the Casimir effect. \em  The energy density between the plates, now known as the \em Casimir energy, \em was calculated to be negative.  The plates essentially  reduce the fluctuations in the gap between them creating negative energy and pressure, which pulls the plates together.  For this reason, negative Casimir energy is associated with an attractive force.  

The Casimir effect in different spacetimes is an important concept in cosmology \cite{MR3448446,MR4133746,MR2651030,MR2455082}, quantum field theory \cite{GRAHAM200249, MR2325409, MR2247334, MR2970520, BORDAG20011}, supergravity \cite{MR1894692, MR884592, MR811397}, superstring theory \cite{KIKKAWA1984357, BINETRUY198968}, hadronic physics \cite{QUEIROZ2005220}, and acoustic scattering \cite{sun2023numerical}. 
The evaluation of the Casimir effect for massless scalar fields (or spinor fields) has been obtained in, e.g.,  \cite{MR1156241} and \cite{PhysRevA80012503}. Moreover, in \cite{strohmaier2013algorithm}, the authors calculate the Casimir energy for several hyperbolic manifolds that include, among others, the Bolza surface. However, the aforementioned spacetimes do not allow singularities.  Hence, they exclude the possibility of exciting and crucial physical objects like Schwarzschild black holes and cosmic strings. Geometrically, these both would create a conical singularity~\cite{cosmicsing}, which is not featured in those geometric settings. 

Nonetheless, the Casimir energy in spacetimes that may have conical singularities has been studied by several authors including \cite{kkcone}, \cite{wedges}, and  \cite{supersymm}. However, the geometric context in the aforementioned works is somewhat restrictive.  Hence there is motivation to understand the Casimir energy in a broader context.  Here we calculate the Casimir energy in two dimensional spacetimes that admit an orbifold structure and may have finite many conical singularities.

A Riemannian orbifold is  singular generalization of a Riemannian manifold which is locally modeled on the quotient of a manifold under a finite group of isometries. The orbifolds were first introduced by Satake~\cite{MR79769} and, in the coming years, became important not only in mathematics, but also in cosmology and physics. 
For example, an ${SU(3)\times SU(2) \times U(1)}$ supersymmetric theory is constructed with an orbifold $\mathbb{S}^1 /(\Z / 2 \Z \times \Z / 2 \Z' )$;  the orbifold fixed points are crucial for the description of supersymmetric Yukawa
interactions~\cite{barbieri2001constrained, arkani2000self}.

Several articles are dedicated to the calculation of the Casimir energy, e.g. \cite{CasimirBrane} or \cite{MR1849734}, in the case of an orbifold. However, in the mentioned articles the authors have closed expressions for the vacuum modes (hat is, closed expressions for the eigenvalues of the Laplace operator). 
However in most cases it is
impossible to have precise formulas for the eigenvalues and consequently
it is a  natural problem to study the  positivity of
the Casimir energy.

To describe the orbifold surfaces in this work, let  $\Gamma$ be a discrete subgroup of the group of orientation-preserving isometries, $PSL_2(\R)$, acting on the hyperbolic upper half plane, $\mathbb H$.  Moreover, assume that the orbifold, $X = \Gamma \backslash \mathbb{H}$, obtained by taking the quotient of the hyperbolic upper half plane with $\Gamma$, is compact. We denote its volume by $\text{vol}(\Gamma \backslash \mathbb{H})$. Then, 
\begin{equation}\label{ineq:eigenvalues}
0 = \lambda_0 < \lambda_1 \leq \lambda_2 \leq \ldots \to \infty	
\end{equation}
 are the eigenvalues of the associated Laplace operator~$\Delta$ acting on $X$. The spectral zeta function,~$\zeta_\Gamma(s)$, of $X$ is defined for $\Rea(s)$  sufficiently large as 
\[\zeta_\Gamma(s) = \sum_{n \in \N} \lambda_n^{-s}.\]

With the help of the Selberg trace formula it is possible to show that the spectral zeta function admits a meromorphic continuation to $s\in \C$. Its value at $s=-1/2$ is referred to as the \textit{Casimir energy}.  Our goal is to give an expression for $\zeta_\Gamma(-1/2)$ in terms of geometric data of the orbifold.  The geometry of the orbifold is determined by the elements of the group $\Gamma$.  These are classed in the following two types. 

\begin{enumerate}
	\item A non-identity element, $R \in \Gamma$, is \textit{elliptic} if it is of finite order.  We note that any cyclic subgroup, $\mathcal{R}$ of finite order in $\Gamma$ is generated by a \textit{primitive elliptic element} $R_0$ of order $m_{\cR} \in \N$. This element $R_0$ may be chosen in $\PSL_2(\R)$ to be conjugate to \[\left( \begin{matrix}
		\cos (\pi / m_\mathcal{R}) & -\sin(\pi / m_\mathcal{R}) \\
		\sin(\pi / m_\mathcal{R}) & \cos(\pi / m_\mathcal{R})
	\end{matrix}\right).\] The angle, $\theta_\mathcal{R} = \pi / m_\mathcal{R}, $
is the smallest positive angle among all such angles determined by the elements of the group generated by $R_0$. We denote the set of all primitive elliptic elements of $\Gamma$ by~$ \{ \mathcal{R} \}_p$.

	\item An element $P \in \Gamma$ is \textit{hyperbolic} if it is $\text{PSL}_2(\R)$-conjugate to 
	\[ \begin{pmatrix} a(P) & 0 \\ 0 & a(P)^{-1} \end{pmatrix}, \]
	such that $1<a(P)$. The norm of $P$ is defined to be $NP := |a(P)|^2$. The element $P$ gives rise to a closed geodesic in $\Gamma \backslash \HH$ which has length $\ell_{P} = \log NP$. We let $k$ be the biggest positive integer such that $P=P_0^k$ for some $P_0 \in \Gamma$. If $k=1$, we say that $P=P_0$ is a \textit{primitive hyperbolic element}.  
 
 We denote the set of $\Gamma$-conjugacy classes of all  hyperbolic elements, respectively primitive hyperbolic elements, by $\{ \mathcal{P}\}$, respectively $\{ \mathcal{P}\}_p$. We additionally note that if $\Gamma$ has no elliptic elements, the set $\{ \mathcal{P}\}$ is in 1-to-1 correspondence with the set of oriented closed geodesics of $X$. 

%
\end{enumerate} 


We note that the compactness of $X$ excludes the possibility that $\Gamma$ contains so-called \textit{parabolic} elements. This is equivalent to saying that there are no elements $\gamma \in \Gamma$ that are $\PSL_2(\R)$-conjugated 
to 
\[ \begin{pmatrix} 1 & x \\ 0 & 1 \end{pmatrix} \]
for some $x \in \R \setminus \{ 0\}$. As many hyperbolic surfaces of interest are not compact, several authors studied values of spectral zeta function for such groups.  For example, in \cite{Hashimoto}, the author investigated certain values of the spectral zeta function in the case of presence of parabolic elements. However, since the presence of parabolic elements causes the surface to be non-compact, this changes the structure of the Laplace spectrum.  In particular it is no longer discrete.  Consequently, a modified approach is required to investigate the energy in that case which shall be the subject of the future work. 



We further recall the Struve function of the second kind and the modified Bessel function of the second kind. 
\begin{defn} \label{def:K_pmbK}
We denote by $\pmb{K}_{j}$ the $j$-th Struve function of the second kind, 
\[ \pmbK_j(z) = \pmb{H}_j(z) - Y_j (z).\]
Here, $\pmb{H}_j$  is the $j$-th Struve function of the first kind as defined in \cite[\S 10.4]{watson} (see also \cite[\S 11.2]{NIST}), and $Y_j$ the Bessel function of the second kind, also known as the Weber Bessel function defined in \cite[\S 3.53]{watson} (see also \cite[\S 10.2]{NIST}).  The modified Bessel function of the second kind of order~$j$ is denoted $K_j$ and defined in \cite[p. 64]{watson} (see also \cite[\S 10.27, 10.31]{NIST}). 
\end{defn}

With these preparations, we may now state our first main result.  

\begin{thm} \label{th:energy}
The Casimir energy of the orbifold $\Gamma \backslash \mathbb H$, 
\begin{align*}
\zeta_\Gamma(-1/2)& =\frac{ \vol(\Gamma \backslash \mathbb{H}) }{\pi} \sum_{n=0}^\infty \frac{n+1}{2^{n+6}} \sum_{k=0}^{n}  (-1)^{k+1}  \binom{n}{k} \frac{ \pmb{K}_{2} \left[ \pi(1 +   k) \right] }{(1+k)^2} \\
	& + \sum_{ \{ \mathcal{R} \}_p }  \sum_{\ell=1}^{m_\mathcal{R}-1} \frac{1}{8 m_{\mathcal{R}} \sin(\tfrac{\pi \ell}{m_\mathcal{R}})}   \sum _{n=0}^{\infty} \frac{1}{2^{n+1}}   \sum _{k=0}^n  (-1)^k  \binom{n}{k} 
	\frac{   \pmb{K}_{1}\left[ (k+\tfrac{ \ell}{m_\mathcal{R}}) \pi \right]} 
	{k+\tfrac{ \ell}{m_\mathcal{R}} } \\
	&  -\frac{1}{4 \pi }\sum_{ \{ \mathcal{P} \}_p} \sum_{n=1}^\infty \tfrac{1}{n}\csch(\tfrac{n \ell_\gamma}{2}) K_{1} (\tfrac{n \ell_\gamma}{2}).
\end{align*}
\end{thm}

The first two lines of the right hand side converge exponentially fast. Moreover, it is possible to evaluate the Struve functions with the help of systems of computer algebra with an arbitrary precision.  Consequently, this form is extremely convenient for calculations.

It is the presence of conical singularities, corresponding to the elliptic elements of the group, that allows for the possibility of \em positive \em Casimir energy, corresponding to a \em repulsive force. \em  Without these elements the Casimir energy is always strictly negative. We show that under a natural assumption on the lengths of closed geodesics, an assumption that is known to hold asymptotically, the Casimir energy may be \em positive. \em
The groups of which we are aware that may give rise to orbifolds with positive Casimir energy are so-called triangle groups. Fix $p,q,r \in \N$ with $\tfrac{1}{p}+\tfrac{1}{q}+\tfrac{1}{r}<1$. Define a $(p,q,r)$\textit{-triangle group} as in \cite[Definition 10.6.3]{Beardon} and denote it by $\Gamma(p,q,r)$. For such values of $p,q,r$, the group $\Gamma(p,q,r)$ is a discrete co-compact subgroup of $\PSL(2, \R)$. The area of $\Gamma(p,q,r) \backslash \mathbb{H}$ is equal to \cite[p. 280]{Beardon}
\begin{equation}\label{eq:volume}
2 \pi \left( 1- (\tfrac{1}{p} + \tfrac{1}{q}+\tfrac{1}{r})\right).	
\end{equation}
There has been a significant amount of research dedicated to 
triangle groups  \cite{suzzi2016figure, marmolejo2020growth, philippe2008groupes, philippe2010rigidite, philippe2009spectre}.

One of the most significant triangle groups is the $(2,3,7)$-triangle group, $\Delta(2,3,7)$.   It is related to a special type  of   \em   surfaces, \em named after Adolf Hurwitz.  A Hurwitz surface, is a compact Riemann surface with precisely 
${84(g-1)}$ automorphisms, where $g$ is the genus of the surface.  This number is maximal by virtue of Hurwitz's theorem on automorphisms \cite{hurwitz1892}.   This   group of automorphisms is called a Hurwitz group.  By uniformization, a Hurwitz surface admits a hyperbolic structure wherein the automorphisms act by isometries.  Such isometries descend from the $(2,3,7)$-triangle group acting on the universal cover $\mathbb H$.

  Here, we aim to show that under a natural assumption on the closed geodesics of the 
  $(2,3,7)$-triangle group orbifold, which is known to hold asymptotically, the Casimir energy is positive.    

\begin{conjecture} \label{conj_plus}
Under the assumption \eqref{eq:equality_for_ell_n}, the Casimir energy of the  $(2,3,7)$-triangle group orbifold $\Delta(2,3,7) \setminus \mathbb{H}$ is bigger or equal than 0.01.
\end{conjecture}

With the standard sign convention, negative Casimir energy physically represents an attracting force, whereas positive Casimir energy physically represents a repelling force \cite{asorey}.  In Lemma \ref{lemma:elliptic_contribution} we show that the first term in the expression for the Casimir energy, $\zeta_\Gamma(-1/2)$ given in Theorem \ref{th:energy} is strictly negative.  It is also apparent that the last term is strictly negative.  The middle term is the contribution of the elliptic elements.  This shows that the Casimir energy is always negative for smooth compact hyperbolic surfaces without conical singularities since they have no elliptic elements.  Moreover, in the case of the $(2,3.7)$-orbifold surface it shows that the presence of conical singularities has a profound effect, to the extent that their contribution to the energy is the dominant term.  We red that for many, perhaps even most, surfaces obtained as a quotient by a $(p,q,r)$ triangle group, the Casimir energy is positive, but we postpone that investigation to future work.  


\subsection{Numerics}
Some of the calculations in this paper were performed  with the help of PARI/GP \cite{PARI2} using a multiple-precision arithmetic with the precision of 500 significant digits. To be more precise, we used it in the proof of Lemma~\ref{le:elliptic}, \eqref{eq:final1} and Table \ref{tab:mytable}. The code is available upon request.

\subsection{Organization} \label{ss:o}
In \S \ref{sec:preliminaries} we recall basic properties of triangle groups and the spectral zeta function, the Selberg trace formula, and standard notation.  We continue in \S \ref{s:elliptic} with the calculation of the orbital integrals arising from the elliptic elements.  One interesting observation that follows from Lemma~\ref{le:growth_lemma} is that as the angle of the elliptic element tends to zero, the contribution to the Casimir energy is positive and tends to infinity on the order of $\theta^{-2}$ for an angle of measure $\theta$.  We then calculate to six significant figures the elliptic contribution to the Casimir energy of the $(2,3,7)$-triangle group orbifold.  In \S 4 we calculate the identity contribution in general and demonstrate an estimate for the $(2,3,7)$-triangle group orbifold in particular.  In \S \ref{s:hyperbolic} we consider the hyperbolic contribution to the Casimir energy in general and then specialize to the case of the $(2,3,7)$-triangle surface.  We follow~\cite{Vogeler} to calculate to six significant figures the contribution from the first 50 primitive closed geodesics.  Next, under assumption \ref{eq:equality_for_ell_n} on the remaining geodesic lengths, we estimate the contribution of all but the first 50 primitive closed geodesics.  We conclude this section with a proof of Conjecture~\ref{conj_plus} under this assumption, noting that the assumption holds asymptotically.  In \S \ref{s:conclude} we conclude with implications and further directions.  

\section*{Acknowledgements} 
JR is grateful to Peter Sarnak for inspiring discussions and to Roger Vogeler for elucidating correspondence.  GZ's research is partially supported
by the Swedish Research Council (VR). KF is partially funded by the Deutsche Forschungsgemeinschaft (DFG, German Research Foundation) under Germany's Excellence Strategy EXC 2044-390685587, Mathematics M\"unster: Dynamics-Geometry-Structure.

\section{Preliminaries}\label{sec:preliminaries}
Here we recall additional facts about the geometry of compact hyperbolic surfaces. In Sections \ref{sec:STF} and \ref{sec:SZF}, we discuss the Selberg trace formula, sketch the proof of the meromorphic continuation of $\zeta_\Gamma(s)$ and obtain the elliptic orbital integrals.



\subsection{Selberg trace formula}\label{sec:STF}
As in \eqref{ineq:eigenvalues}, we let $\{ \lambda_n\}_{n \in \N_0}$ be the eigenvalues of the Laplace operator acting on $X$. 
It is convenient to introduce a sequence of numbers $r_n$ such that the eigenvalues 
\beq \lambda_n = 1/4+r_n^2, \textrm{ for } n \in \N_0. \label{eq:r_n} \eeq 
Then $r_0$ must be $\pm \frac i 2$, and $r_n$ is real for $n \geq 1$. To state the Selberg trace formula as in \cite{Hejhal1}, assume that the function $r \mapsto h(r)$ is analytic on $|\Ima(r)| \leq \frac 1 2 + \delta$ for some $\delta > 0$.  Assume further that $h$ is even, that is $h(-r) = h(r)$, and that $h$ satisfies an estimate $|h(r)| \leq M (1+\Rea(r))^{-2-\delta}$ for a constant~$M$.  We define the Fourier transform of $h$ to be 
\beq g(u) = \frac{1}{2\pi} \int_\R h(r) e^{-iru} dr. \label{eq:fouriertransform} \eeq 
Then, with this setup, the Selberg trace formula is the following identity \cite[p. 351-352]{Hejhal1}
\begin{align}
 \sum_{n \geq 0} h(r_n)=
 & \frac{ \vol(\Gamma \backslash \HH)}{4\pi} \int_\R r h(r) \tanh(\pi r) dr \nonumber \\
&+  \sum_{\{ \mathcal{P} \}} \frac{ \log NP_0}{NP^{1/2} - NP^{-1/2}} g(\log(NP)) \nonumber \\
&+\sum_{ \{ \mathcal{R} \}_p } \sum_{\ell=1}^{m_\mathcal{R}-1} \frac{1}{2 m_\mathcal{R} \sin (\tfrac{\pi \ell}{m_\mathcal{R}})} \int_\R \frac{e^{-2r \tfrac{\pi \ell}{m_\mathcal{R}}}}{1+e^{-2\pi r}} h(r) dr. \label{eq:Selberg_trace_formula}
\end{align}
The sums and integrals in the above expression are all absolutely convergent.  We note that if one compares the above identity to \cite[p. 351-352]{Hejhal1}, the representation we have here is the trivial representation, so the traces appearing in \cite{Hejhal1} are all equal to one.\footnote{The Selberg trace formula has unfortunately appeared incorrectly in the literature in at least two occasions of which we are aware.  We have taken care to verify that this is the correct expression as in \cite{Hejhal1}. In \cite{Hejhal1}, the author uses  the negative of our Laplace operator, but that does not change the values of $r_n$.}  



\subsection{Spectral zeta function}\label{sec:SZF}
In \cite[(6.10) and (6.11)]{FloydCasimir} and \cite[(3)]{MR369286}, the respective  authors study the meromorphic continution of $\zeta_\Gamma(s)$ to $s \in \C$. To avoid the zero in the denominator of the first summand, corresponding to the eigenvalue $\lambda_0 = 0$, they choose $\varepsilon>0$ and introduce
\begin{equation}\label{eq:zeta_via_Mellin_transform}
\zeta_{\Gamma,\varepsilon}(s) = \sum_{n \in \N_0} \frac{1}{(\lambda_n+\varepsilon)^s} = \frac{1}{\Gamma(s)} \int_{0}^{\infty} t^{s-1} \sum_{n \in \N_0} e^{-t ( \lambda_n  + \varepsilon)} dt.
\end{equation}
As the next step, they consider the function 
\[ h(r) = e^{-t (r^2 + 1/4+\varepsilon)}\] and apply the Selberg trace formula, \eqref{eq:Selberg_trace_formula}, to this function to express 
\[ \sum_{n \in \N_0} e^{-t ( \lambda_n  + \varepsilon)}\] in terms of geometric data of $X$. 
The following step is to substitute aforementioned sum into \eqref{eq:zeta_via_Mellin_transform} to  obtain $\zeta_{\Gamma,\varepsilon}(s)$. 
Finally, $\zeta_\Gamma(s)$ is obtained from a limiting procedure by letting  $\varepsilon$ go to $0$. Repeating their proof with the only modification that now we have to take elliptic elements into consideration, we obtain  for $\Rea(s) < 0$,
\begin{align} \label{eq:first_line_in_zeta}
  \zeta_\Gamma(s) 	& =  
  \frac{\vol(\Gamma \backslash \mathbb{H})}{8(s-1)} \int_\R (\tfrac{1}{4}+r^2)^{1-s} \sech^2 (\pi r) dr  \nonumber \\
	& + \frac{(4 \pi)^{-1/2}}{\Gamma(s)}  \sum_{  \{ \mathcal{P} \}_p} \sum_{n=1}^\infty ( \ell_\gamma / n)^{1/2} \csch(\tfrac{n \ell_\gamma}{2}) (n \ell_\gamma)^s K_{1/2-s} (\tfrac{n \ell_\gamma}{2}) \nonumber \\
	& + \sum_{ \{ \mathcal{R} \}_p } \sum_{\ell=1}^{m_\mathcal{R}-1} \frac{1}{2 m_\mathcal{R} \sin (\tfrac{\pi \ell}{m_\mathcal{R}})} \int_\R \frac{e^{-2r \tfrac{\pi \ell}{m_\mathcal{R}}}}{1+e^{-2\pi r}} (\tfrac{1}{4}+r^2)^{-s} dr.
\end{align}

One can obtain the same result \textit{formally}, that is, non-rigorously, by taking $h(r) = (1/4+r^2)^{-s}$; of course, in that case $h$ does not satisfy the growth condition for ${\Rea(s)<0}$. For such $h$, the left hand side of the Selberg trace formula,~\eqref{eq:Selberg_trace_formula}, \text{formally} coincides with the spectral zeta function, as each summand reads $h(r_n) = (1/4+r_n^2)^{-s} = \lambda_n^{-s}$.  Although this may be a useful heuristic, the derivation following \cite[(6.10) and (6.11)]{FloydCasimir} and \cite[(3)]{MR369286} is \em fully rigorous.  \em 

\subsection*{Notation}
We recall the following notation: 
\begin{itemize}
	\item $f \le_{a, b, c, ...} g$ means $\exists C >0$ that depends only on the (finitely many) parameters $a, b, c, \ldots$ such that $f \leq C g$,
	\item $f \lesssim g$ means $\exists C$ (independent of any parameters) such that $f \leq C g$, 
 \item a function $f(x)$ is $\mathcal O (g(x))$ as $x \to 0$ if there exist $C, \varepsilon > 0$ such that $|f(x)| \leq C |g(x)|$ for all $x \in (0, \varepsilon)$, 
	\item $\Gamma(\cdot)$ is the Gamma function, $\Gamma(\cdot, \cdot)$ is the (upper) incomplete Gamma function.
\end{itemize}

\section{Elliptic contribution} \label{s:elliptic}
In this section, we demonstrate an identity that we use to obtain an 
expression for the contribution of elliptic elements to the spectral zeta function in terms of special functions.  This identity is  of independent interest as it may be useful for other calculations due to its rapid convergence.  Here, we use it to evaluate the contribution of the elliptic elements in $\Delta(2,3,7) \backslash \mathbb{H}$ to its Casimir energy.  

\begin{lemma}\label{lemma:elliptic_contribution}
Let 
     $C > 0$, $D \ge 0$, and $s \in \C$. 
Then 	\begin{align}\label{eq:expression_of_an_easy_integral_in_terms_of_struveH}
		&\int_{0}^{\infty}  \frac{e^{- C y}}{e^{-D  y} + 1} (1+y^2)^{-s} dy \nonumber \\
		& = \sqrt{\pi } 2^{{-s} -1/2} \Gamma \left(1-s\right) \sum _{n=0}^{\infty} 2^{-n-1}  \sum _{k=0}^n    \binom{n}{k}   \frac{ (-1)^k  \pmb{K}_{{-s}+1/2}(C+D k)}{(C+D k)^{{-s}+1/2}}.
	\end{align}
	Above, $\pmb{K}_{-s+1/2}$ is the Struve function of the second kind. In particular, for $s=-1/2$, the right hand side of  \eqref{eq:expression_of_an_easy_integral_in_terms_of_struveH} becomes 
	\[
	\pi  \sum _{n=0}^{\infty} 2^{-n-2}  \sum _{k=0}^n    \binom{n}{k}   \frac{ (-1)^k  \pmb{K}_{1}(C+D k)}{C +D k}.
	\]
 
The series converges exponentially fast; more precisely, the 
absolute value of the difference between the left hand side of \eqref{eq:expression_of_an_easy_integral_in_terms_of_struveH} and the right hand side, restricted to $n \in \{0, N \}$, is bounded by
\begin{equation}\label{eq:error_in_evaluation}
 \frac{1}{2^N} \int_0^1 \frac{(1+\log(t)^2)^{-\Rea(s)}  t^{C-1}}{(t^D + 1 )} dt.   
\end{equation}
 For $s=-1/2$ in particular, we have the following bound:
 \begin{equation}\label{eq:bound_for_12}
     \frac{1}{2^N}\frac{\pi  C  \pmb{K}_1(C)+4}{2 C^2}.
 \end{equation}

\end{lemma}

\begin{proof}
 We make the change of variables $t=e^{-y}$ and rewrite the integral:
	\begin{align}\label{eq:integral_easy_elliptic}
		\int_{0}^{\infty}  \frac{e^{- C y}}{e^{-D  y} + 1} (1+y^2)^{-s}  dy = \int_{0}^{1}  \frac{t^{C-1}}{t^{D} + 1} (1+\log(t)^2)^{-s} dt.
	\end{align}
	For $|x-1|<2$,  
	\begin{align}\label{eq:series_expansion_of_1/x+1}
		\frac{1}{x+1} = \sum_{n=0}^{\infty} (-1)^n (x-1)^n 2^{-n-1}   = \sum_{n=0}^{\infty} \sum _{k=0}^n (-1)^k 2^{-n-1} x^k \binom{n}{k}.
	\end{align}
	If $x = t^D \in [0,1]$, the series above converges uniformly.  
 Moreover, reversing the substitution, 
 \[ \int_0 ^\infty e^{-Cy} (1+y^2)^{-s} dy = \int_0 ^1 t^{C-1} (1+\log(t)^2)^{-s} dt.\]
 Since $C>0$, the $L^1$ norm of the function $x \mapsto x^{C-1}(1+\log^2(x))^{-s}$ is finite on $[0, 1]$.
 This allows us to substitute \eqref{eq:series_expansion_of_1/x+1} into \eqref{eq:integral_easy_elliptic} and exchange the summation and the integration to obtain that \eqref{eq:integral_easy_elliptic} is equal to 
	\[
	\sum_{n=0}^{\infty} 2^{-n-1} \sum _{k=0}^n (-1)^k  \binom{n}{k} \int_{0}^{1}  t^{C+ Dk-1} (1+\log(t)^2)^{{-s}}    dt.
	\]
	 As a consequence of \cite[(11.5.2)]{NIST},
the sum above is equal to 
\[
\sqrt{\pi } 2^{{-s} -1/2} \Gamma \left(1-s\right) \sum _{n=0}^{\infty} 2^{-n-1}  \sum _{k=0}^n    \binom{n}{k}   \frac{ (-1)^k  \pmb{K}_{{-s}+1/2}(C+D k)}{(C+D k)^{{-s}+1/2}},
\]
that concludes the proof.

	It remains  to prove that the convergence is exponentially fast. 
	Observe that 
	\[
	\frac{1}{x+1}	-\sum _{n=0}^{N-1}  (-1)^n (x-1)^n 2^{-n-1} =\frac{(1-x)^N}{2^N (x+1)}.
	\]
 
	We obtain that the absolute differences between the right and the left hand sides of \eqref{eq:expression_of_an_easy_integral_in_terms_of_struveH} is bounded from above by  
	\begin{align}\label{ineq:error_bound}
	&	\left| 	\int_{0}^{1}  \frac{(1-t^D)^N}{2^N (t^D +1)} t^{C-1} (1+\log(t)^2)^{-s} dt \right| \\ 
 & \leq \frac{1}{2^N} \int_0^1 \frac{(1+\log(t)^2)^{-\Rea(s)}  t^{C-1}}{(t^D + 1 )} dt. \nonumber
	\end{align} 
	The right hand side decays exponentially fast as $N \to \infty$ therewith proving the exponential convergence in  \eqref{eq:expression_of_an_easy_integral_in_terms_of_struveH}.	 
 We note that for $C>0$ and $D>0$, we can estimate the right hand side of \eqref{ineq:error_bound} by 
 \[
\frac{1}{2^N} \int_0^1 (1+\log(t)^2)^{1/2}  t^{C-1} dt = \frac{1}{2^N}\frac{\pi  C  \pmb{K}_1(C)+4}{2 C^2}.
 \]
\end{proof}

\begin{lemma}\label{eq:contribution_from_elliptic_elements_to_casimir_energy}
	The contribution of elliptic elements to the Casimir energy is equal to 
\begin{align*}
	\sum_{ \{ \mathcal{R} \}_p } & \sum_{\ell=1}^{m_\mathcal{R}-1} \frac{1}{8 m_{\mathcal{R}} \sin(\tfrac{\pi \ell}{m_\mathcal{R}})}   \sum _{n=0}^{\infty} \frac{1}{2^{n+1}}   \sum _{k=0}^n  (-1)^k  \binom{n}{k} 
	  \frac{   \pmb{K}_{1}\left[ (k+\tfrac{ \ell}{m_\mathcal{R}}) \pi \right]}
	{k+\tfrac{ \ell}{m_\mathcal{R}} }.
\end{align*}
\end{lemma}

\begin{proof}
	We recall from \eqref{eq:first_line_in_zeta} that the contribution from elliptic elements to $\zeta_\Gamma(s)$ is  equal to
	\[
	\sum_{ \{ \mathcal{R} \}_p } \sum_{\ell=1}^{m_\mathcal{R}-1} \frac{1}{2 m_\mathcal{R} \sin (\tfrac{\pi \ell}{m_\mathcal{R}})} \int_\R \frac{e^{-2r \tfrac{\pi \ell}{m_\mathcal{R}}}}{1+e^{-2\pi r}} (\tfrac{1}{4}+r^2)^{-s} dr,
	\]
	where $\{ \mathcal{R} \}_p$ and $m_\mathcal{R}$ are defined as in \S \ref{sec:intro}. 
Making the substitution $t=2r$, the integral 
 \begin{equation}\label{eq:integrals_transform}
 \int_\R \frac{e^{-2r \tfrac{\pi \ell}{m_\mathcal{R}}}}{1+e^{-2\pi r}} (\tfrac{1}{4}+r^2)^{-s} dr = 4^s \int_0 ^\infty \frac{e^{-t \pi \ell /m_\cR}}{1+e^{-\pi t}} (1+t^2)^{-s} dt.      
 \end{equation}
Since $\ell< m_\cR$, we may apply Lemma \ref{lemma:elliptic_contribution} to conclude that this integral is equal to 
\[ 4^s \sqrt \pi 2^{-s-1/2} \Gamma(1-s) \sum_{n \geq 0} 2^{-n-1} \sum_{k=0} ^n (-1)^k {n \choose k} \frac{ \pmbK_{-s+1/2}(\pi \ell/m_\cR + \pi k)}{(\pi \ell/m_\cR + \pi k)^{-s+1/2}}.\]
Setting $s=-1/2$ we obtain 
\[ \frac 1 4 \sum_{n = 0} ^\infty 2^{-n-1} \sum_{k=0} ^n (-1)^k {n \choose k} \frac{ \pmbK_1(\pi \ell/m_\cR + \pi k)}{\ell/m_\cR + k}\]
that concludes the proof.

\end{proof}

\subsection{Elliptic elements in triangle groups}\label{sec_class_elements}
Let $\Gamma(p, q, r)$ be the $(p,q,r)$-triangle group. In
 \cite{philippe2011determination}, lengths, $\ell_1, \ell_2, \ell_3$, of the first three geodesics for  $q \ge p \ge r \ge 3$ are given as:
\begin{align*}
	\ell_1 &= 2 \arcosh \left( 2 \cos \tfrac{\pi}{r} \cos  \tfrac{\pi}{p} + \cos \tfrac{\pi}{q} \right),\\
	\ell_2 &= 2 \arcosh \left( 2 \cos \tfrac{\pi}{q} \cos  \tfrac{\pi}{r} + \cos \tfrac{\pi}{p} \right), \\
	\ell_3 &= 2 \arcosh \left( 2 \cos \tfrac{\pi}{p} \cos \tfrac{\pi}{q} + \cos \tfrac{\pi}{r} \right).
\end{align*}
 We note that the  group $\Delta(p,q,r)$ has (up to conjugacy in $\Delta(p,q,r)$) three cyclic subgroups of finite orders  with $m_\mathcal{R} \in \{ p, q, r \}$ (that statement is also true when $r=2$).  For more details, we refer to \cite[p. 163]{Iwaniec}, \cite[pp. 98-99]{Kubota}.  

In the following Lemma we show that for large values of $m_R$ and, respectively, small values of  $\theta = \frac{\pi}{m_R}$, the contribution of elliptic elements to the Casimir energy becomes large. 
\begin{lemma} \label{le:growth_lemma} 
For $\theta \searrow 0$, 
\begin{equation}\label{eq:growth_lemma}
\int_{0}^{\infty}  \frac{e^{- \theta y}}{e^{- \pi  y} + 1} (1+y^2)^{1/2} dy \gtrsim \theta^{-2}.
\end{equation}
\end{lemma}	
\begin{proof}
Using \eqref{eq:integral_easy_elliptic}, we evaluate the left hand side of \eqref{eq:growth_lemma} from below by 
\[
\frac{1}{2} \int_0^1 t^{\theta-1} (1+\log(t)^2)^{1/2} dt = \frac{\pi  \pmb{K}_1(\theta)}{4 \theta }.
\]
Around $\theta=0$, 
by \cite[11.2.1]{NIST},  
\[ \pmb{H}_1(\theta) = \mathcal O (\theta^2), \quad \theta \to 0.\]
By \cite[10.7.4]{NIST}, 
\[ Y_1(\theta) = \frac{-1}{2\pi \theta} + \mathcal O(1), \quad \theta \to 0. \]
Since 
\[ \pmbK_1(\theta) = \pmb{H}_1(\theta) - Y_1(\theta),\]
we therefore have 
\[ \frac{\pi  \pmb{K}_1(\theta)}{4 \theta } = \frac{1}{8 \theta^2} + \mathcal O(1), \quad \theta \to 0.\]
\end{proof}

\begin{lemma} \label{le:elliptic}
	The elliptic contribution to the Casimir energy of $\Delta(2,3,7) \setminus \mathbb{H}$ rounded to six decimal places is equal to 0.875676.  
\end{lemma}
\begin{proof}
	We note that the only elliptic elements in this group are those of order $2$, $3$ and $7$ \cite[Proposition 2.1]{Vogeler}.
Thus, we are interested in the sum
	\begin{align}\label{eq:to_evalute_in_lemma34}
		&\sum_{m_\cR=2,3,7} \sum_{k=1}^{m_\cR-1} \frac{1}{2 m_\cR \sin(k \frac{\pi}{m_\cR})} \int_\R \frac{e^{-2 k \frac{\pi}{m_\cR}  y}}{e^{-2 \pi y}+1} (\tfrac14+ y^2)^{1/2} dy. 
\end{align}
We can use Lemma \ref{lemma:elliptic_contribution} to evaluate integrals in  \eqref{eq:to_evalute_in_lemma34}. 
In order to choose $N$ that would provide a sufficiently accurate  approximation, we recall  \eqref{eq:integrals_transform}. Its evaluation is equivalent to the evaluation of the integral in Lemma \ref{lemma:elliptic_contribution} for $D = \pi$ and various values of 
\[C \in \left\{ \frac{\pi \ell}{m_\cR}, \ell \in \{1, \ldots, m_\cR -1 \}, m_\cR  \in \{2, 3, 7\} \right\}.\]
We also note that the right hand side  of  \eqref{ineq:error_bound} is a decreasing function of $C$, thus it will suffice to find the error for $C = \pi / 7$.  Further, we note that for $N=100$, the error in the evaluation of \eqref{eq:integrals_transform} can be bounded by  
 \begin{equation}\label{eq:bound_for_121}
  4^{-1/2}   \frac{1}{2^N}\frac{\frac{\pi^2}{7} \pmb{K}_1(\frac{\pi}{7})+4}{2 (\pi/7)^2} <  10^{-29} .
 \end{equation}
Given that we only want to evaluate the elliptic contribution up to six significant digits, this certainly suffices, but we need to take into account the accumulation of errors that will appear once we find the total sum of on the order of $10^4$ summands.  Consequently, that will slightly decrease the precision to the order of  $10^{-25}$.  Moreover, 
\[
\max_{m_\cR = 2, 3, 7} \max_{k = 1, \ldots, m_\cR} \left| \frac{1}{2 m_{\mathcal{R}} \sin(k \frac{\pi}{m_{\mathcal{R}}})} \right| < 1.
\]
Thus, using  Lemma \ref{lemma:elliptic_contribution} with $N=100$ would be sufficient to evaluate \eqref{eq:to_evalute_in_lemma34} up to  six significant figures, and we obtain the value 0.875676.

\end{proof}

\section{Identity contribution} \label{s:identity}
It is possible to rewrite an identity contribution to $\zeta_\Gamma(s)$ as an infinite sum of special functions in the same spirit as we did for the elliptic contribution in Lemma \ref{lemma:elliptic_contribution}. 
\begin{lemma}
	The identity contribution to the spectral zeta function,  
 \[\frac{\vol(\Gamma \backslash \mathbb{H})}{8(s-1)} \int_\R (\tfrac{1}{4}+r^2)^{1-s} \sech^2 (\pi r) dr\]
is equal to 
 \[ \frac{\vol(\Gamma \backslash \mathbb{H})}{8(s-1)} 4^s \pi^{s-1} \Gamma(2-s)  \sum_{n=0} ^\infty \sum_{k=0} ^n \frac{(n+1)(-1)^{k}}{2^{n+3/2+s}} { n \choose k} \frac{\pmbK_{3/2-s}(\pi + \pi k)}{(1+k)^{3/2-s}}.\]
This sum converges exponentially fast.  In particular, for $s=-1/2$ this  
 is equal to 
\[ - \frac{\vol(\Gamma \backslash \mathbb{H})}{\pi } \sum_{n=0} ^\infty \sum_{k=0} ^n \frac{ (n+1)(-1)^{k}}{2^{n+6}}{ n \choose k}  \frac{ \pmbK_2 (\pi + \pi k)}{(1+k)^2}. \]

\end{lemma}

\begin{proof}
Observe that for any constant $D>0$,  
\[ \sech^2(D r) = \frac{4 e^{-2 D r}}{(1 + e^{-2 D r})^2}.\]

Then
\[ \int_\R (\tfrac{1}{4}+r^2)^{1-s} \sech^2 (D r) dr
=  \int_\R (\tfrac{1}{4}+r^2)^{1-s}\frac{4 e^{-2 D r}}{(1 + e^{-2 D r})^2} dr  \] 
\[ = 4^{s-1} \int_\R (1+(2r)^2)^{1-s}\frac{4 e^{-2 D r}}{(1 + e^{-2 D r})^2} dr =  4^s  \int_0 ^\infty (1+y^2)^{1-s} \frac{e^{-Dy}}{(1+e^{-Dy})^2} dy.\]

Above we used the substitution $y=2r$ and the fact that the integrand is even. For $|x-1|<2$,   (compare with  \eqref{eq:series_expansion_of_1/x+1}) 
\[ \frac{1}{(1+x)^2} = - \frac{d}{dx} \frac{1}{x+1}  = \sum_{n = 0} ^\infty (-1)^{n} (x-1)^n (n+1) 2^{-n-2}\]
\[ = \sum_{n= 0} ^\infty \sum_{k=0} ^n (-1)^{k} (n+1) 2^{-n-2} x^k {n \choose k}. \]
We use this together with the absolute convergence of the integral (since $D>0$) to obtain (with $x=e^{-Dy}$) 
\[4^s \sum_{n= 0} ^\infty \sum_{k=0} ^n (-1)^{k} \frac{n+1}{ 2^{n+2}}  {n \choose k} \int_0 ^\infty (1+y^2)^{1-s}e^{-(D+kD)y} dy. \]
By \cite[11.5]{NIST}, this is equal to 
\[4^s \sqrt \pi \Gamma(2-s) 2^{1/2-s} \sum_{n= 0} ^\infty \sum_{k=0} ^n (-1)^{k} \frac{n+1}{ 2^{n+2}}  {n \choose k} (D+Dk)^{s-3/2} \pmbK_{3/2-s}(D+Dk)  \]

Setting $D=\pi$ this becomes 
\[ 4^s \pi^{s-1} \Gamma(2-s)  \sum_{n=0} ^\infty \sum_{k=0} ^n \frac{(n+1)(-1)^{k}}{2^{n+3/2+s}} { n \choose k} \frac{\pmbK_{3/2-s}(\pi + \pi k)}{(1+k)^{3/2-s}}. \]
Recalling the factor of 
\[ \frac{\vol(\Gamma \backslash \mathbb{H})}{8(s-1)} \] 
completes the first statement of the Lemma.  
Estimates analogous to the proof of Lemma \ref{lemma:elliptic_contribution} show the exponential rate of convergence.  Specializing to $s=-1/2$ we obtain that the identity contribution to the Casimir energy is 
\[ - \frac{\vol(\Gamma \backslash \mathbb{H})}{\pi } \sum_{n=0} ^\infty \sum_{k=0} ^n \frac{ (n+1)(-1)^{k}}{2^{n+6}}{ n \choose k}  \frac{ \pmbK_2 (\pi + \pi k)}{(1+k)^2}. \]

\end{proof}

For our purposes, we do not need the full precision of the expression in the preceding Lemma.  As we will show in Lemma \ref{le:identity}, specialized to for the $(2,3,7)$-triangle group orbifold, the rough estimate we obtain in Lemma \ref{lemma:identity_in_Hurwitz} shows that the contribution from the identity element is bounded within the interval of approximate size $10^{-3}$.  That is sufficient for our application to estimate the Casimir energy and show that under a natural assumption on the lengths of the closed geodesics of the surface, it is positive.

\begin{lemma}\label{lemma:identity_in_Hurwitz}
	The identity contribution to the Casimir energy is contained in the interval 
		\[ \left(  -  \frac{2 \vol(\Gamma \backslash \mathbb{H})}{45 \pi }, -  \frac{\vol(\Gamma \backslash \mathbb{H})}{36 \pi }\right). \] 
\end{lemma}
\begin{proof}
	We use \cite[Theorem 1]{Randal} to rewrite the contribution from the identity element as ($s=-\frac 12$)
	\begin{align*}
	  & \left.  \frac{\vol(\Gamma \backslash \mathbb{H})}{8(s-1)}  \int_{-\infty}^{\infty} (\tfrac{1}{4}+r^2)^{1-s} \sech^2 (\pi r) dr \right|_{s=-1/2} \\
		& \ \ \ \ \ \ = -\frac{\vol(\Gamma \backslash \mathbb{H})}{12} \int_{-\infty}^{\infty} (\tfrac{1}{4}+r^2)^{\tfrac{3}{2}} \sech^2 (\pi r) dr.
	\end{align*}
	We evaluate this from below as follows:
	\begin{align*}
		-\frac{ \vol(\Gamma \backslash \mathbb{H})}{12} \int_{-\infty}^{\infty} (\tfrac{1}{4}+r^2)^{\tfrac{3}{2}} \sech^2(\pi r) dr 
   \ge -\frac{\vol(\Gamma\backslash\mathbb H)}{96} \int_{-\infty}^{\infty} (1+4 r^2)^{2} \sech^2(\pi r) dr.
	\end{align*}
 
We shall calculate this integral using \cite[3.527.3.12]{gr} which states that 
		\[ \int_0 ^\infty \frac{x^{b-1}}{\cosh^2 x} dx= 2^{2-b} (1-2^{2-b}) \Gamma(b) \zeta(b-1), \quad \text{Re} (b) > 0, \quad b \neq 2,\]
	where $\zeta$ denotes the Riemann zeta function.
 Then, we note that 
	\begin{align*}
	 \int_{-\infty} ^\infty & (4r^2+1)^2 \sech^2(\pi r) dr = 2 \int_0 ^\infty (16r^4+8r^2+1)\sech^2(\pi r) dr\\ 
& = 2 \left( \int_0 ^\infty \left( \frac{16 s^4}{\pi^4} \sech^2(s) + \frac{8 s^2}{\pi^2} \sech^2(s) + \sech^2(s) \right) \frac{ds}{\pi}\right) \\
& = 2 \left( \frac{16}{\pi^5} 2^{2-5} (1-2^{2-5}) \Gamma(5) \zeta(4) + \frac{8}{\pi^3} 2^{2-3} (1-2^{2-3}) \Gamma(3)\zeta(2) + \frac 1 \pi \right) \\
& = 2\left( \frac{7}{15 \pi} + \frac{2}{3\pi} + \frac{1}{\pi} \right) = \frac{64}{15 \pi}. 
	\end{align*}
	This shows that the contribution from the identity element is bigger than 
	\[
	-\frac{\vol(\Gamma \backslash \mathbb{H})}{96} \cdot  \frac{64}{15 \pi} = -  \frac{2 \vol(\Gamma \backslash \mathbb{H})}{45 \pi }.
	\]
	
With the same idea we can evaluate the contribution of the identity from above as well by noting that 
	\begin{align*}
&	-\frac{\vol(\Gamma \backslash \mathbb{H})}{12} \int_{-\infty}^{\infty}  (\tfrac{1}{4}+r^2)^{\tfrac{3}{2}} \sech^2 (\pi r) dr \\ 
		&\leq -\frac{2 \vol(\Gamma \backslash \mathbb{H})}{96}  \int_0^{\infty } \left(4 r^2+1\right) \text{sech}^2(\pi  r) \, dr \\
		& =  -\frac{\vol(\Gamma \backslash \mathbb{H})}{36 \pi }.
	\end{align*}
	\end{proof}
\begin{lemma} \label{le:identity} 
The identity contribution to the Casimir energy of the $(2,3,7)$-triangle group orbifold belongs to $[-0.00211640, -0.00132275]$.   
\end{lemma}
\begin{proof}
	We use Lemma \ref{lemma:identity_in_Hurwitz} and the formula for the area of the surface, \eqref{eq:volume}, to get the  upper estimate 
	\[
	-\frac{2 \pi \left(1-\left(\tfrac{1}{2}+\tfrac{1}{3}+\tfrac{1}{7}\right)\right)}{36 \pi } = -\frac{1}{756} \approx -0.00132275
	\]
	and the lower estimate
	\[
	-\frac{4 \pi \left(1-\left(\tfrac{1}{2}+\tfrac{1}{3}+\tfrac{1}{7}\right)\right)}{45 \pi } =  -\frac{2}{945} \approx -0.00211640.
	\]
\end{proof}

\section{Contribution from hyperbolic elements}	\label{s:hyperbolic}
The hyperbolic contribution to $\zeta_\Gamma(s)$ (see \eqref{eq:first_line_in_zeta})  is equal to 
\begin{align*}	 & \frac{(4 \pi)^{-1/2}}{\Gamma(s)}  \sum_{ \{ \mathcal{P} \}_p} \sum_{n=1}^\infty ( \ell_\gamma / n)^{1/2} \csch(\tfrac{n \ell_\gamma}{2}) (n \ell_\gamma)^s K_{1/2-s} (\tfrac{n \ell_\gamma}{2}).
\end{align*}
We recall that $\sum_{ \{ \mathcal{P} \}_p}$ denotes the summation over all conjugacy classes of primitive hyperbolic elements.
Specialized at $s=-1/2$, this reads
\begin{align}\label{eq:H_contribution}
	 & \frac{ (4 \pi)^{-1/2}}{\Gamma(-1/2)} \sum_{ \{ \mathcal{P} \}_p} \sum_{n=1}^\infty (\ell_\gamma / n)^{1/2} \csch(\tfrac{n \ell_\gamma}{2}) (n \ell_\gamma)^{-1/2} K_{1} (\tfrac{n \ell_\gamma}{2}) \nonumber  \\
	&= -\frac{1}{4 \pi }\sum_{ \{ \mathcal{P} \}_p} \sum_{n=1}^\infty (\ell_\gamma / n)^{1/2} \csch(\tfrac{n \ell_\gamma}{2}) (n \ell_\gamma)^{-1/2} K_{1} (\tfrac{n \ell_\gamma}{2}) \nonumber \\
	&= -\frac{1}{4 \pi }\sum_{ \{ \mathcal{P} \}_p} \sum_{n=1}^\infty \tfrac{1}{n}\csch(\tfrac{n \ell_\gamma}{2}) K_{1} (\tfrac{n \ell_\gamma}{2}).
\end{align}

In the following Lemma we will take advantage of the fact that Vogeler \cite[p.32]{Vogeler} obtained explicit expressions for lengths of the  first 50 primitive closed geodesics of the $(2,3,7)$ orbifold surface.  With this we can calculate their contribution to the Casimir energy quite accurately. 


\begin{lemma}\label{lemma:hyperbolic_contribution}
	The contribution to the Casimir energy of  $\Delta(2,3,7) \backslash \mathbb{H}$ from the first 50 primitive geodesics rounded to six decimal places is equal to $-0.5680851$.
\end{lemma}
\begin{proof}
The proof of this lemma uses explicit formulas for the first geodesics calculated in  \cite[p. 32]{Vogeler}. We note  
\begin{align*}
	b&=\tfrac{1}{3} \left(\sqrt{3 (\cot^2(\tfrac{\pi }{7})-3)}+\sqrt{3} \cot (\tfrac{\pi }{7})\right), \\
	A &= \left(
	\begin{smallmatrix}
		\cos (\tfrac{\pi }{3}) & \sin (\tfrac{\pi }{3}) \\
		-\sin (\tfrac{\pi }{3}) & \cos (\tfrac{\pi }{3}) \\
	\end{smallmatrix}
	\right), \quad 
	B =\left(
	\begin{smallmatrix}
		\cos \left(\tfrac{\pi }{7}\right) & b \sin \left(\tfrac{\pi }{7}\right) \\
		-b^{-1}\sin \left(\tfrac{\pi }{7}\right) & \cos \left(\tfrac{\pi }{7}\right) \\
	\end{smallmatrix}
	\right), \\
	R & =  A^{-1}B, \quad L=B.
\end{align*}
Above, $A$ and $B$ are respective rotations of order $3$ and $7$ that generate the group $\Delta(2,3,7)$. Thus, each hyperbolic element may be represented as a product of $R$ and $L$.

On \cite[p. 32]{Vogeler}, the author calculates a finite portion of the length spectrum.  To do this, he develops a combinatorial approach which leads to a classification of the conjugacy classes of hyperbolic elements of $\Delta(2,3,7)$, arranged by length. 
For the convenience of the reader, we present approximate lengths of primitive closed geodesics together with representatives of the corresponding conjugacy classes in Table~\ref{tab:mytable}. We note that lengths of closed geodesics can be expressed as a finite combination of cotangents, cosines, sines, square roots and logarithms (as $A$ and $B$ are functions of such) and thus may be calculated with an arbitrary precision.
 
The situation with multiplicities is a bit subtle. 
Let $\gamma \in \Gamma = \Delta(2,3,7)$ be a primitive hyperbolic element and denote the length of the corresponding closed geodesic by~$\ell_\gamma$. Then, $\gamma^{-1}$ is also a hyperbolic element and $\ell_{\gamma^{-1}} = \ell_\gamma$.  As described in \cite[p. 24]{Vogeler}, by changing the $R$'s and $L$'s in the representation of $\gamma$, one obtains a hyperbolic element $\gamma^*$ with $\ell_{\gamma^*} = \ell_\gamma$.   
So we therefore have \cite[p. 24]{Vogeler}
\[
\ell_{\gamma} = \ell_{\gamma^{-1}} = \ell_{\gamma^*} = \ell_{(\gamma^*)^{-1}}.
\]
We let  
$
s(\gamma)
$
be the number of distinct conjugacy classes among
$ \{ \gamma \}, \{ \gamma^{-1} \}$, $\{ \gamma^{*} \}$ and $\{ (\gamma^{*})^{-1} \}$. 
For example, let $\gamma$ be a hyperbolic element and assume that $\{ \gamma \} = \{ \gamma^{-1} \}$ and $\{ \gamma^* \} = \{ (\gamma^*)^{-1} \}$, but $\{ \gamma \}  \neq \{ \gamma^* \}$; in this case we say  $s(\gamma) = 2$. If, on the other hand, $\{ \gamma \} = \{ \gamma^{-1} \} = \{ \gamma^* \} = \{ (\gamma^*)^{-1} \}$, then $s(\gamma) = 1$. If it turns out that the conjugacy classes $\{ \gamma \}$, $ \{ \gamma^{-1} \}$, $\{ \gamma^* \}$ and $ \{ (\gamma^*)^{-1} \}$ are pairwise different, then $s(\gamma)=4$.

We additionally note that $s(\gamma)$ is not necessarily a multiplicity of the geodesic length. It might happen that $\ell_\gamma = \ell_{\gamma'}$, but at the same time, 
\[ \gamma' \not \in \{\gamma \} \cup \{\gamma^{-1} \} \cup \{ \gamma^* \} \cup \{ (\gamma^*)^{-1} \}. \]
In this case, the multiplicity of the geodesic length is bigger than $s(\gamma)$. Among the geodesics that we take into account, this situation happens exactly once: in Table \ref{tab:mytable}, one finds two geodesics of approximate lengths $5.2889$, that are, however, listed separately. 

To sum it up, for each hyperbolic element $\gamma \in \Gamma$, there are $s(\gamma)$ closed geodesics of length $\ell_{\gamma}$ corresponding to distinct conjugacy classes among $ \{ \gamma \}$, $ \{ \gamma^{-1} \}$, $ \{ \gamma^* \}$ and $ \{ (\gamma^*)^{-1} \}$.
 This implies that  Table \ref{tab:mytable} contains the first 50 primitive closed geodesics of the $(2,3,7)$-triangle group orbifold. We denote by $\ell_n$ the length of the $n$-th primitive closed geodesic; thus, $\ell_1 \approx 0.98, \ell_3 = \ell_4 \approx 2.13$. We define
\[
A(\gamma) =  -\frac{1}{4\pi}\sum_{n=1}^\infty\frac{s(\gamma)}{n}\csch\left(\frac{n\ell_\gamma}{2}\right) K_{1} \left(\frac{n\ell_\gamma}{2}\right).
\]
The total contribution of hyperbolic elements to the spectral zeta function is equal to the sum of $A(\gamma)$ where $\gamma$ ranges over all primitive hyperbolic elements. Summing up all of $A(\gamma)$ from Table \ref{tab:mytable}, we obtain a value of $-0.5680851,$ rounding to seven decimal places. 

\begin{table}[htbp]
	\centering
	\begin{tabular}{|c|c|c|c|}
		\hline
		$ \ell_{\gamma} \approx  $ & $s(\gamma)$ & $\gamma$ &  $A(\gamma) \approx $ \\ 
  \hline 0.98399 & 1&$R.L$ &-0.288955    \\ 
  \hline 1.73601 & 1&$R.R.L.L$ &-0.064746   \\ 
  \hline 2.13111 & 2&$R.L.R.L.L$ &-0.069526   \\ 
  \hline 2.66193 & 2&$R.L.R.R.L.L$ &-0.032848   \\ 
  \hline 2.89815 & 2&$R.L.L.R.R.L.L$ &-0.024028   \\ 
  \hline 3.15482 & 2&$R.L.R.L.R.L.L$ &-0.017289   \\ 
  \hline 3.54271 & 1&$R.L.R.R.L.R.L.L$ &-0.0053429   \\ 
  \hline 3.62732 & 2&$R.L.R.L.R.R.L.L$ &-0.0096416   \\ 
  \hline 3.80470 & 2&$R.L.R.R.L.R.R.L.L$ &-0.0077879   \\ 
  \hline 3.93595 & 2&$R.L.R.L.L.R.R.L.L$ &-0.0066608   \\ 
  \hline 4.15197 & 2&$R.L.R.L.R.L.R.L.L$ &-0.0051635   \\ 
  \hline 4.20181 & 1&$R.L.L.R.R.L.R.R.L.L$ &-0.0024355   \\ 
  \hline 4.39146 & 2&$R.L.R.R.L.L.R.R.L.L$ &-0.0039068   \\ 
  \hline 4.48926 & 2&$R.L.R.L.R.R.L.R.L.L$ &-0.0034894   \\ 
  \hline 4.60473 & 2&$R.L.R.L.R.L.R.R.L.L$ &-0.0030555   \\ 
  \hline 4.65401 & 2&$R.L.L.R.R.L.L.R.R.L.L$ &-0.0028877   \\ 
  \hline 4.76043 & 2&$R.L.R.L.R.R.L.R.R.L.L$ &-0.0025571   \\ 
  \hline 4.84180 & 4&$R.L.R.L.L.R.L.R.R.L.L$ &-0.0046617   \\ 
  \hline 4.93876 & 2&$R.L.R.L.R.L.L.R.R.L.L$ &-0.0020879   \\ 
  \hline 5.01322 & 2&$R.L.R.L.L.R.L.L.R.R.L.L$ &-0.0019192   \\ 
  \hline 5.14068 & 2&$R.L.R.L.R.L.R.L.R.L.L$ &-0.0016622   \\ 
  \hline 5.20802 & 2&$R.L.R.L.L.R.R.L.R.R.L.L$ &-0.0015409   \\ 
  \hline 5.28890 & 2&$R.L.R.L.R.L.L.R.L.R.L.L$ &-0.0014072   \\ 
  \hline 5.28890 & 2&$R.L.R.R.L.R.L.L.R.R.L.L$ &-0.0014072   \\ 
  \hline 5.35146 & 2&$R.L.R.L.R.R.L.L.R.R.L.L$ &-0.0013120   \\ 
  \hline 5.42680 & 1&$R.L.R.L.R.R.L.R.L.R.L.L$ &-0.00060298   \\ 
  \hline 5.45943 & 2&$R.L.R.L.R.L.R.R.L.R.L.L$ &-0.0011628   \\ \hline 	\end{tabular}
	\caption{The first several lengths of primitive hyperbolic closed geodesics of the $(2,3,7)$-triangle group orbifold together with their representations and their contribution to the Casimir energy.}
	\label{tab:mytable}
\end{table}

\end{proof}

We denote by 
\[ \cN_L = \# \{ \ell_\gamma \leq L \} \]
the number of primitive hyperbolic geodesics $\gamma$ (counted with multiplicity) of length $\ell_\gamma$ less or equal than $L$. The prime geodesic theorem~\cite{sarnak} states that 
\[ \lim_{L \to \infty} \frac{\log L \cN_L}{L} = 1.\]
Consequently, it follows that when we enumerate these lengths (counting multiplicity) as $\ell_j$, we obtain  
\[ \lim_{j \to \infty} \frac{\ell_j}{\log(j) + \log \log (j)} = 1.\]
We propose that it is reasonable to assume,  and we note that this inequality holds for $j < 51$,  \begin{equation}\label{eq:equality_for_ell_n}
	\ell_j \geq \log(j) + \log \log (j), \quad \forall j \ge 51. \end{equation}



	\begin{lemma}\label{lemma:we_hope_it_holds}
		Under the assumption that \eqref{eq:equality_for_ell_n} holds, the contribution from  all but the first 50 hyperbolic elements is greater than or equal to -0.293867. 
	\end{lemma}
We need a small technical lemma before proceeding to the proof of Lemma~\ref{lemma:we_hope_it_holds}:

\begin{lemma} \label{le:jest}
For $j \geq 16$ and $n \in \mathbb{N}$,
\[
j^n \log^{n+1/2} j \leq  (j^n \log^n j - 1)(\log j + \log 
\log j)^{1/2}.
\] 
\end{lemma}
\begin{proof}
We note that the statement of the Lemma follows (after the change of variables $x=\log j$) from 
\begin{align*}
	\sqrt{x+\log x} \leq &e^{nx} x^n \left(\sqrt{x+\log x}- \sqrt x\right)  \iff   \\ 
	\frac{\sqrt{x+\log x}}{\sqrt{x+\log x} - \sqrt x} \leq & e^{nx} x^n \iff  \\ 
	\frac{x + \log x + \sqrt x \sqrt{x+\log x}}{\log x} \leq & e^{nx} x^n. 
\end{align*}	
	Since $n \geq 1$, $x e^x \leq x^n e^{nx} $ for $x>0$.  Moreover, for $\log x \geq 1$, which further guarantees that $x\geq 1$, and for $e^x \geq 2 + \sqrt 2$, we have 
	\[ \frac{x + \log x + \sqrt x \sqrt{x+\log x}}{\log x} \leq x+1+\sqrt{x} \sqrt{2x} \leq (2+\sqrt 2) x \leq x e^x.\]
	Recalling that $x= \log j$, it is enough to assume that $j \geq e^e$.
	
\end{proof}

	\begin{proof}[Proof of Lemma \ref{lemma:we_hope_it_holds}]
		We split the sum in three parts:
		\begin{enumerate}
			\item $n=1$, $j \in [51, 10^7]$,
			\item $n=1$, $j \geq 10^7 + 1$,
			\item $n \ge 2$.
		\end{enumerate}

Since $\csch$ and $K_1$ are decreasing functions on $(0, \infty)$ \cite[10.37]{NIST}, then under the assumption \eqref{eq:equality_for_ell_n} we obtain 
\begin{align*}
 \csch\left( \frac{n \ell_j}{2} \right) & K_1 \left( \frac{n \ell_j}{2} \right) \\
  & \leq \csch \left( \frac{n (\log j + \log \log j)}{2} \right) K_1 \left( \frac{n (\log j + \log \log j)}{2} \right).
\end{align*}
We use this to obtain upper bounds for the sums 
\beq  B_1 &=&  \sum_{j =51}^{10^7} \frac{1}{4\pi}  \csch\left( \frac{\ell_j}{2} \right) K_1 \left( \frac{\ell_j}{2} \right) , \nn \\  
B_2 &=& \sum_{j = 10^7 + 1} ^\infty \frac{1}{4\pi} \csch\left( \frac{\ell_j}{2} \right) K_1 \left( \frac{\ell_j}{2} \right), \nn \\ 
B_3 &=& \sum_{n=2} ^\infty \sum_{j=51} ^\infty \frac{1}{4\pi n}\csch\left( \frac{n\ell_j}{2} \right) K_1 \left( \frac{n\ell_j}{2} \right). \nn 
 \eeq 
  
For $B_1$, we obtain by explicit calculation  the estimate 
	\begin{align}		 
		B_1 & \leq 
		\sum_{j=51}^{10^7}	\frac{\text{csch}
			\left[\frac{ \log (j)+\log\log(j)}{2}  \right] 
			K_1 \left[\frac{\log (j)+\log\log(j)}{2}  \right]
		}{4 \pi  } \nonumber \\
	& \approx 0.138415. \label{eq:final1}
	\end{align}

  
For $B_2$ and $B_3$, we note that for any $j \ge 10^7$,  \cite[(10.37.1)]{NIST} implies 
		\begin{align*}
			&\frac{\text{csch}
				\left[\frac{ \log (j)+\log\log(j)}{2} n \right] 
				K_1 \left[\frac{\log (j)+\log\log(j)}{2} n \right]
			}{4 \pi  n}
			\\
			&\leq \frac{\text{csch}
				\left[\frac{\log (j)+\log\log(j)}{2} n \right]
				K_{\frac{3}{2}}
				\left[\frac{\log (j)+\log\log(j)}{2} n \right]
			}{4 \pi  n}.
		\end{align*}
By \cite[10.39.2]{NIST}, 
\[ K_{1/2}(z) = K_{-1/2}(z) = \sqrt{ \frac{\pi}{2z}} e^{-z}.\]
By \cite[10.25(ii), 10.29(i)]{NIST} 
\[ e^{3\pi i /2} K_{3/2}(z) = - e^{-i\pi/2} K_{-1/2}(z) + 2 e^{i\pi /2} \frac{\partial }{ \partial z}K_{1/2}(z).\]
We therefore obtain 
\[ -i K_{3/2} (z) = i \sqrt{ \frac{\pi}{2z}} e^{-z} - i \sqrt{\frac {\pi}{2z}} z^{-1} e^{-z} - 2 i  \sqrt{ \frac{\pi}{2z}} e^{-z}\]
\[ \implies K_{3/2} (z) = \sqrt{\frac{\pi}{2z}}e^{-z}(z^{-1}+1) .\]
Consequently, since $\csch(z) = \sinh(z)^{-1} = \frac{2}{e^z-e^{-z}}$, 
\[ \csch(z) K_{3/2} (z) = \sqrt{\frac{2\pi}{z}} \frac{z^{-1}+1}{e^{2z}-1} = \sqrt{\frac{2\pi}{z}} \frac{1+z}{ze^{2z} - z}.\]
We then obtain by setting $z=n(\log j + \log \log j)/2$ and dividing by $4\pi n$
\begin{align}
& \frac{\csch	\left[\frac{\log (j)+\log\log(j)}{2} n \right]
K_{\frac{3}{2}}	\left[\frac{\log (j)+\log\log(j)}{2} n \right]}{4 \pi  n}  \nonumber \\ 
&   = \frac{ n (\log (j)+\log \log (j))+2}{2 \sqrt{\pi } n \left( j^n \log^n (j) -1\right)  n^{3/2} (\log (j)+\log \log (j))^{3/2}}. \label{eq:intermediate_eva1235}    
\end{align}

		We recall that $n \ge 1$ and note that for $j \geq j_N$ we can estimate the denominator in \eqref{eq:intermediate_eva1235} from above as 
		\[
		n (\log (j)+\log \log (j))+2 \leq A_{n, j_N} n (\log (j)+\log \log (j)),
		\]
		where 
		\[
		A_{n, j_N} = 1 + \frac{2}{n(\log(j_N) +  \log\log(j_N))}.
		\]
		Thus we can evaluate \eqref{eq:intermediate_eva1235} from above by 
		\begin{align*}
			&
		A_{n, j_N}	 n \frac{
				\log (j)+\log \log (j)
			}{2 \sqrt{\pi } n \left( j^n \log^{n}(j) -1\right) [ n (\log (j)+\log \log (j))]^{3/2}} \\
			& \leq 
			\frac{
		A_{n, j_N}	
			}{2 \sqrt{\pi } n^{3/2} (j^n \log^n(j)-1) (\log (j)+\log \log (j))^{1/2}} \\
		& 
  \leq \frac{A_{n, j_N}}{2 \sqrt{\pi}} \frac{ j^{-n} \log^{-n-1/2}(j)}{ n^{3/2}}.
		\end{align*}
In the last step we used Lemma \ref{le:jest}.	We therefore obtain the estimate for~$B_2$
 \[ \sum_{j=10^7+1} ^\infty \csch\left( \frac{\ell_j}{2} \right) K_1 \left( \frac{\ell_j}{2} \right) \leq  \frac{1+\frac{2}{\log \left(10^7\right)+\log \left(\log \left(10^7\right)\right)}}{2 \sqrt{\pi }} \sum_{j=10^7 + 1} j^{-1} \log^{-3/2} j.\]
For $n=1$ and $j_N = 10^7$, we evaluate 
	\[
	\frac{A_{1, 10^7}}{2 \sqrt{\pi}} = \frac{1+\frac{2}{\log \left(10^7\right)+\log \left(\log \left(10^7\right)\right)}}{2 \sqrt{\pi }} \approx 0.311949
	\]
and 
 \[
	 \sum _{j=10^7 + 1}^{\infty } j^{-1} \log ^{-1-\frac{1}{2}}(j) \leq  \int_{10^7}^{\infty } \frac{\log ^{-1-\frac{1}{2}}(j)}{j} \, dj = \frac{2}{\sqrt{\log (10^7)}} \approx 0.498165.
	\]
This gives the estimate 
\beq B_2 \leq 0.155402. \label{eq:final2} \eeq 

	For $n\ge 2$ and $j_N = 51$, we obtain 
\[ \sum_{n=2} ^\infty \sum_{j=51} ^\infty \frac{1}{4\pi n}\csch\left( \frac{n\ell_j}{2} \right) K_1 \left( \frac{n\ell_j}{2} \right) \leq \sum_{n = 2} ^\infty \sum_{j = 51} ^\infty \frac{A_{n, 51} j^{-n} \log^{-n-1/2} (j) }{2 \sqrt{\pi} n^{3/2}} .  \]
 
We evaluate for $n \geq 2$ 
	\[
	\frac{A_{n, 51}}{2 \sqrt{\pi}} \leq \frac{1+\frac{1}{\log \left(51\right)+\log \log \left(51\right)  } }{2 \sqrt{\pi }} \approx 0.335311. 
	\]
	Using the definition of the polylogarithm $\text{Li}_{3/2}$ of order $3/2$,
 we obtain 
 \[
 \text{Li}_{\tfrac{3}{2}}(\tfrac{1}{j \log(j)}) = \sum_{n=1}^\infty \frac{j^{-n} \log(j)^{-n}}{n^{3/2}},
 \]
 and thus for each $j \geq 51$
\[ \sum_{n=2} ^\infty \frac{j^{-n} \log(j)^{-n-1/2}}{n^{3/2}} = \log^{-1/2}(j) \left(\text{Li}_{\tfrac{3}{2}}(\tfrac{1}{j \log(j)}) - j^{-1} \log^{-1}(j) \right).  \]

We  calculate
\begin{align}
\sum_{j=51}^{\infty} \sum _{n=2}^{\infty } \frac{j^{-n} \log ^{-n-\frac{1}{2}}(j)}{n^{3/2}} &= \sum_{j=51}^{\infty} \frac{ \left(\text{Li}_{\tfrac{3}{2}}(\tfrac{1}{j \log(j)}) - j^{-1} \log^{-1}(j) \right)} {\log^{1/2}(j)} \nonumber \\
& = \sum_{j=51}^{\infty} \frac{j \log (j) \text{Li}_{\frac{3}{2}}\left(\frac{1}{j \log (j)}\right)-1}{j \log ^{\frac{3}{2}}(j)} \nonumber \\
& \leq \sum_{j=51}^{\infty} \frac{j \log (j) \text{Li}_{1}\left(\frac{1}{j \log (j)}\right)-1}{j \log ^{\frac{3}{2}}(j)} \nonumber \\
& = \sum_{j=51}^{\infty} \frac{-j \log (j) \log\left(1-\frac{1}{j \log (j)}\right)-1}{j \log ^{\frac{3}{2}}(j)}. \label{eq:polylog_inequality}
\end{align}
Above, both the first and the third equality follow from the definitions of $\text{Li}_1$ and $\text{Li}_{3/2}$. 
We note that for $x \geq 50$, the following inequality holds:
\[ 0 < -\log \left(1-\frac{1}{x} \right) \le \frac{1}{x} + \frac{1}{1.9 x^2},\]
thus for $j>50$,
\[ 0 < -\log \left(1-\frac{1}{j \log(j)} \right) \le \frac{1}{j \log(j)} + \frac{1}{1.9 j^2 \log(j)^2}.\]
 With this we estimate  \eqref{eq:polylog_inequality} from above by 
\begin{align*}
\sum_{j=51}^\infty \frac{1}{1.9 j^2 \log ^{\frac{5}{2}}(j)} & \leq \frac{1}{1.9} \int_{50}^\infty \frac{1}{j^2 \log(j)^{5/2}} dj
\\&= \frac{1}{1.9} \left( \frac{2}{150 \log^{3/2} (50)} - \frac{4}{150\sqrt{\log(50)}} + \frac{4 \Gamma ( \log(50), \tfrac{1}{2} )}{3}  \right)  \\
& \approx 0.000224. 
\end{align*}
Above $\Gamma(a, s)=\int_a^\infty t^{s-1} e^{-t}dt$ in the
incomplete Gamma function. We further note that the first equality follows from the calculation 
\[
\int \frac{1}{j^2 \log ^{\frac{5}{2}}(j)} \, dj = -\frac{2}{3j \log ^{\frac{3}{2}}(j)} + \frac{4}{3j \sqrt{\log (j)}}- \frac{4  \Gamma \left(\log (j), \frac 1 2 \right)}{3} + C.
\]

Thus, we obtain the estimate 
\beq B_3 \leq 0.335311\cdot 0.000224 \approx 0.000075. \label{eq:final3} \eeq 
Consequently, summing \eqref{eq:final1}, \eqref{eq:final2} and \eqref{eq:final3} we obtain that 
\[ B_1 + B_2 + B_3 \leq 0.293892.\]

Recalling the minus sign in front of the hyperbolic contribution thereby completes the proof of its lower bound. 
	\end{proof}

\begin{proof}[Proof of Conjecture \ref{conj_plus} under the assumption \eqref{eq:equality_for_ell_n}] 
By Lemma \ref{le:identity}, the identity contribution to the Casimir energy is at least $-0.0022$.  By Lemma \ref{lemma:hyperbolic_contribution} the contribution from the first 50 primitive hyperbolic geodesics is, up to six decimal places, $-0.5680851$.  By Lemma \ref{le:elliptic} the contribution of the elliptic elements is, up to six decimal places, $0.875676$. By Lemma \ref{lemma:we_hope_it_holds} the contribution from all but the first 50 hyperbolic elements is at least $-0.293867$.  We therefore obtain a lower bound of the Casimir energy 
\[ \zeta_\Gamma(-1/2) \geq 0.875676 -0.002116-0.5680851 -0.293867 = 0.0116079. \]

\end{proof} 

\section{Concluding remarks} \label{s:conclude}
It is well known that the Casimir energy of a hyperbolic orbifold surface depends on the geometry of the surface as this follows from the representation of the spectral zeta function through the Selberg trace formula \cite{Hejhal1}.  Physically, the surface may be used to represent a quantum field theory.  Conjecture \ref{conj_plus} would indicate that the Casimir energy can be attractive or repulsive depending on the geometry of the orbifold.  In particular, without conical singularities, the energy is negative (attractive) and with singularities it may in fact be positive (repulsive).  We reasonably expect to be able to prove the conjecture, but this will require not only the asymptotic behavior of the lengths, which is well known \cite{sarnak} but explicit lower bounds for the lengths.  One can obtain a crude lower bound via volume growth considerations, but we reasonably expect it is possible to obtain a bound that would be sufficient to prove the conjecture.  Moreover, we expect that the hyperbolic elements in other $(p, q, r)$ triangle groups may admit a description in the spirit of \cite{Vogeler}, so that we may be able to prove that for many corresponding orbifold surfaces, the Casimir energy is also positive (repulsive).  This would indicate that the conical singularities profoundly influence the Casimir energy and Casimir effect.  If the orbifold represents a certain quantum field theory, what are the physical implications of such a repulsive Casimir effect?  Perhaps this would be interesting for physicists to consider further and develop experimental tests \cite{sci_amer}.


\end{document}